\documentclass[letterpaper, 10 pt, conference]{ieeeconf}  

\newif\ifonlineapp 
\onlineapptrue
\pdfoutput=1

\IEEEoverridecommandlockouts     

\usepackage{graphics} 
\usepackage{epsfig} 

\usepackage{amsmath}
\usepackage{booktabs} 
\usepackage{amsfonts}
\usepackage{times}
\usepackage{cite}
\usepackage[utf8]{inputenc} 
\usepackage[T1]{fontenc}    
\usepackage{bm, Shortcuts_OPT}

\let\labelindent\relax
\usepackage{enumitem}

\allowdisplaybreaks

\title{\LARGE \bf
On the Role of Data Homogeneity in Multi-Agent Non-convex Stochastic Optimization
}

\author{Qiang Li and Hoi-To Wai$^{1}$
\thanks{$^{1}$Qiang Li and Hoi-To Wai are with System Engineering \& Engineering Management, Faculty of Engineering,
        The Chinese University of Hong Kong, Shatin District, Hong Kong SAR. Emails:~\url{liqiang@se.cuhk.edu.hk}, \url{htwai@se.cuhk.edu.hk}}
}

\newtheoremstyle{exampstyle}
  {1.5\topsep} 
  {1.5\topsep} 
  {} 
  {} 
  {\bfseries} 
  {.} 
  {.5em} 
  {} 

\theoremstyle{exampstyle}

\newtheoremstyle{otherstyle}
  {1.5\topsep} 
  {1.5\topsep} 
  {\itshape} 
  {} 
  {\bfseries} 
  {.} 
  {.5em} 
  {} 

\theoremstyle{otherstyle}

\newtheorem{Corollary}{Corollary}

\newtheorem{Assumption}{A\!}

\makeatletter
\let\NAT@parse\undefined
\makeatother
\usepackage[colorlinks]{hyperref}
\hypersetup{
  colorlinks   = true, 
  urlcolor     = blue, 
  linkcolor    = red, 
  citecolor    = red   
}
\usepackage{cleveref}

\begin{document}

\maketitle
\thispagestyle{empty}
\pagestyle{empty}

\begin{abstract}
This paper studies the role of data homogeneity on multi-agent optimization. Concentrating on the decentralized stochastic gradient ({\sf DSGD}) algorithm, we characterize the transient time, defined as the minimum number of iterations required such that {\sf DSGD} can achieve the comparable performance as its centralized counterpart. When the Hessians for the objective functions are identical at different agents, we show that the transient time of {\sf DSGD} is ${\cal O}( n^{4/3} / \rho^{8/3} )$ for smooth (possibly non-convex) objective functions, where $n$ is the number of agents and $\rho$ is the spectral gap of connectivity graph. This is improved over the bound of ${\cal O}( n^2 / \rho^4 )$ without the Hessian homogeneity assumption. Our analysis leverages a property that the objective function is twice continuously differentiable. Numerical experiments are presented to illustrate the essence of data homogeneity to fast convergence of {\sf DSGD}.
\end{abstract}

\section{Introduction}
Consider a system of $n$ agents which are connected on a network. We are concerned with the following multi-agent stochastic optimization problem:
\beq \textstyle \label{eq:opt}
\min_{ \prm \in \RR^d } f(\prm) := (1/n) \sum_{i=1}^n f_i( \prm ),
\eeq 
where $f_i: \RR^d \to \RR$, $d \in \NN$, is the local stochastic objective function held by agent $i$, $i=1,\ldots, n$. The gradient of each function $f_i(\prm)$ is assumed to be Lipschitz continuous with respect to (w.r.t.) the decision variable $\prm \in \RR^d$ while the function itself is possibly non-convex. 
In addition, the $n$ agents communicate with neighboring agents through an undirected graph  as they tackle \eqref{eq:opt} in a collaborative fashion.

The multi-agent optimization problem \eqref{eq:opt} has wide applications in control and machine learning (ML) \cite{nedic2009distributed, bottou2018optimization, chang2020distributed}. We concentrate on the distributed ML application. The function $f_i(\prm)$ models the mismatch of the decision variable $\prm \in \RR^d$ with respect to the \emph{local data} held by the $i$th agent. Concretely, we consider objective function of the form
\beq \label{eq:stocfct}
f_i ( \prm ) = \EE_{ z_i \sim {\sf B}_i } [ \ell ( \prm; z_i ) ] ,
\eeq 
where the data distribution ${\sf B}_i$ is defined on the space ${\sf Z}$ as it describes the data at the $i$th agent, and $\ell : \RR^d \times {\sf Z} \to \RR$ is the loss function. 
We remark that in distributed ML, a common setting is to assume \emph{homogeneous data} such that ${\sf B}_i = {\sf B}_j$ for all $i,j$. The latter models a scenario with $f_i(\prm) \equiv f_j(\prm)$ where agents observe independent and identically distributed (i.i.d.) samples \cite{Zhang2015CommunicationEfficientDO, Reddi2016AIDEFA}. Note this is in contrast to the non-i.i.d.~setting where \emph{heterogeneous data} is observed \cite{wang2021field}.

This paper focuses on tackling (\ref{eq:opt}) via stochastic distributed first-order algorithms. In the basic setting, each agent carries out the optimization of a local estimate of a stationary solution to \eqref{eq:opt} using noisy gradients of its local objective function $f_i(\prm)$. The latter is assumed to be unbiased estimates of $\grd f_i ( \prm )$ with bounded second order moment. Particularly, the distributed stochastic gradient ({\sf DSGD}) method is proposed in \cite{sundhar2010distributed} (also see \cite{nedic2009distributed}) which combines network average consensus with stochastic gradient updates. Despite its simple structure, {\sf DSGD} is shown to be efficient  theoretically and empirically in tackling large-scale machine learning problems. In particular, \cite{lian2017decentralized} showed that {\sf DSGD} achieves a `linear speedup' where the asymptotic convergence rate approaches that of a centralized SGD ({\sf CSGD}) algorithm with a minibatch size of $n$, i.e., with reduced variance. 

However, the convergence rate of {\sf DSGD} can be severely affected by the network size $n$, the mixing rate (a.k.a.~spectral gap) of the connection graph $\rho \in (0,1]$ (see A\ref{ass: graph}). In fact, \cite{lian2017decentralized} demonstrated that the linear speedup of {\sf DSGD} can be guaranteed only when the iteration number exceeds a \emph{transient time} of ${\cal O}( n^2 / \rho^4 )$ iterations, which can be undesirable for system with many agents; also see the recent work \cite{pu2021sharp} which focused on strongly convex optimization problems. Note that we have $\rho = \Theta(1/n^2)$ {for ring graph, $\rho=\Theta(1/n)$ for 2d-torus graph, see \cite{aldous1995reversible}}. 
The study of the transient time is important where it is closely linked to the communication efficiency in distributed optimization.  

Recent works have sought to speed up the convergence of distributed stochastic optimization by designing new and sophisticated algorithms. Examples include \cite{uribe2021dual, scaman2017optimal} which apply multiple communication steps per iteration, D-GET, GT-HSGD \cite{sun2020improving, xin2021hybrid} which combine gradient tracking with variance reduction, EDAS \cite{huang2021improving} which utilizes a similar idea to EXTRA \cite{shi2015extra} but characterizes an improved transient time for strongly convex problems. We also mention that recent works have combined compressed communication in distributed optimization, e.g., \cite{frasca2009average,koloskova2019decentralized}, whose techniques are complementary to the above works.

This paper is motivated by the successes of {\sf DSGD} in practice shown in various works \cite{lian2017decentralized,assran2019stochastic,vlaski2021distributed}. We depart from the prior studies
and inquire the following question:
\begin{center} 
\emph{Can {\sf DSGD} achieve fast convergence with a shorter transient time than ${\cal O}(n^2/\rho^4)$?}
\end{center}
We provide an affirmative answer to the above question through studying the role of \emph{data homogeneity} in distributed stochastic optimization.  
Our key finding is that when the data held by agents are (close to) homogeneous such that the Hessians are close, i.e., $\grd^2 f_i( \prm ) \approx \grd^2 f_j( \prm) $ for any $i,j$ and $\prm \in \RR^d$, then the transient time of {\sf DSGD} can be significantly shortened.
To summarize, our contributions are:
\begin{itemize}[leftmargin=*]
    \item Under the Hessian homogeneity assumption and a second order smoothness condition for the objective function $f_i$, we show that the transient time of {\sf DSGD} can be improved to ${\cal O}(n^{4/3}/\rho^{8/3})$ from ${\cal O}(n^2/\rho^4)$ in \cite{lian2017decentralized}. Our result highlights the role of data homogeneity in the (fast) convergence of {\sf DSGD} which may explain the latter's efficacy in practical large-scale machine learning.
    \item We introduce new proof techniques for finding tight bounds in the convergence of distributed stochastic optimization. Importantly, we demonstrate how to extract accelerated convergence rates when the Hessians of objective functions are Lipschitz. This leads to a set of high order inequalities with fast-decaying errors.
    \item To verify our theorems, we conduct numerical experiment on a toy binary classification problem with linear models. We empirically demonstrate that data homogeneity is a key factor affecting the (fast) convergence of {\sf DSGD} through a controlled comparison between {\sf DSGD} with homogeneous and heterogeneous data.
\end{itemize}
To our best knowledge, this is the first analysis to demonstrate accelerated convergence rate without modifying the simple structure of {\sf DSGD}. Our result provides evidence for the good performance of {\sf DSGD} in practice.

\vspace{+0.1cm}

\noindent {\bf Notations:} Throughout this paper, we use the following notations:
$\norm{\cdot}$ is the vector $\ell_2$ norm or the matrix spectral norm depending on the argument, $\norm{\cdot}_{F}$ is the matrix Frobenius norm, and ${\bm 1}$ is the all-one column vector in $\RR^n$. 
We set 
$f^\star \eqdef \min_{\prm\in\RR^d} f(\prm) > -\infty$ as the optimal objective value of \eqref{eq:opt}.  The subscript-less operator $\EE [\cdot]$ denotes the total expectation taken over all randomnesses in the operand.

\section{Problem Statement and Assumptions}

Consider a multiagent network system whose communication is represented by an undirected graph ${\cal G}=({\cal N}, {\cal E})$, where ${\cal N}=[n] = \{1,\ldots,n \}$ is the set of agents and ${\cal E}\subseteq {\cal N}\times {\cal N}$ denotes the set of edges between the communicating agents. Note that $(i,i) \in {\cal E}$ as self loops are included in ${\cal G}$. Every agent $i \in {\cal N}$ can directly receive and send information only from its neighbors $\{ j : (i,j) \in {\cal E} \}$. 

Furthermore, the graph ${\cal G}$ is endowed with a symmetric, weighted adjacency matrix (a.k.a.~mixing matrix) ${\bm W} \in \RR^{n \times n}$ such that $W_{ij} = W_{ji} > 0$ if and only if $(i,j) \in {\cal E}$; otherwise $W_{ij} = W_{ji} = 0$. Moreover, we assume that
\begin{Assumption}\label{ass: graph}
The matrix ${\bm W} \in \RR^{n \times n}$ is doubly stochastic, i.e., ${\bm W}\mathbf{1}={\bm W}^\top\mathbf{1}=\mathbf{1}$. There exists a constant ${\rho} \in(0,1]$ and a projection matrix ${\bm U} \in \RR^{n \times (n-1)}$ which can be represented as $\boldsymbol{I}-\frac{1}{n} \mathbf{1} \mathbf{1}^{\top}={\bm U} {\bm U}^{\top}$ such that $\left\|{\bm U}^{\top} \bm{W} \bm{U} \right\|_{2} \leq 1-{\rho}$.
\end{Assumption}
\noindent The above is a common assumption for the connected graph ${\cal G}$. For instance, the mixing matrix ${\bm W}$ satisfying A\ref{ass: graph} can be constructed using the Metropolis-Hasting weights \cite{aldous1995reversible}.

To tackle \eqref{eq:opt}, a classical algorithm is the decentralized stochastic gradient descent ({\sf DSGD}), whose recursion at iteration $t \geq 0$ can be described as 
\begin{align}\label{algo}
    \prm_i^{t+1} = \sum_{j=1}^n W_{ij} \prm_j^t - \gamma_{t+1} {\grd} \ell ( \prm_i^t ; z_i^{t+1} ) , ~ i=1,\ldots,n,
\end{align}
where $z_i^{t+1} \sim {\sf B}_i$ is a sample drawn independently from the distribution ${\sf B}_i$, and $\gamma_{t+1} > 0$ is the step size.

The {\sf DSGD} algorithm \eqref{algo} is a \emph{gossip} type algorithm where information is spread along the edges on the communication graph ${\cal G}$. 
At iteration $t$, the local iterate $\prm_i^t$ held by agent $i$ is communicated to the neighboring nodes $j$. 
Each agent $i$ performs a consensus update by computing an average of its local iterate as well as the neighbors' iterates via the mixing matrix ${\bm W}$. Subsequently, the agent draws a stochastic gradient estimate $\grd \ell( \prm_i^t; z_i^{t+1} )$ to perform a gradient step.

\subsection{Convergence of {\sf DSGD}: Basic Results}
We first discuss a basic convergence result for {\sf DSGD} which is derived under a general setting that does not require the data across agents in the model \eqref{eq:opt} to be homogeneous. Notice that our result below is akin to the analysis in \cite{lian2017decentralized}. 

Our result depends on a few standard assumptions, which have been used in prior works such as \cite{lian2017decentralized}, as follows: 
\begin{Assumption} \label{ass: smooth}
For any $i = 1,...,n$, there exists $L \geq 0$ such that
\beq
\| \grd f_i( \prm' ) - \grd f_i(\prm )\| \leq L \| \prm' - \prm \|,~\forall~\prm', \prm \in \RR^d.
\eeq
\end{Assumption}
\noindent The above assumes that the gradient of each local objective function is $L$-Lipschitz continuous. Note that under the above assumption, $f_i(\prm)$ can be non-convex. 

\begin{Assumption}\label{ass:SecOrdMom}
For any $i=1,\ldots, n$ and fixed $\prm \in \RR^d$, it holds $\EE_{z_i \sim {\sf B}_i} [ \grd \ell( \prm ; z_i ) ] = \grd f_i (\prm)$ and there exists $\sigma \geq 0$ with
\begin{align}
\EE_{z_i \sim {\sf B}_i} [ \| \grd \ell( \prm ; z_i ) - \grd f_i( \prm ) \|^2 ] \leq \sigmatwo^2 . \label{A4-1}
\end{align}
\end{Assumption}
\begin{Assumption}\label{ass:hete}
For any $i=1,\ldots,n$, there exists $\varsigma \geq 0$ such that:
\begin{align}
& \| \grd f( \prm ) - \grd f_i( \prm) \| \leq \varsigma,~\forall~\prm \in \RR^d. \label{eq:hete_grad} 
\end{align}
\end{Assumption}
\noindent In the above, A\ref{ass:SecOrdMom} states that the stochastic gradient estimates are unbiased and have bounded variance. Meanwhile, A\ref{ass:hete} assumes that the gradients  of the component function, $\grd f_i(\prm)$, have bounded distance from the gradient of the average function, $\grd f(\prm)$. 
Notice that the scalar $\varsigma$ measures the amount of data homogeneity (via gradient). If $\varsigma = 0$, then $f_i(\prm), f_j(\prm)$ differ only by a constant; see \cite{lian2017decentralized}. 


It will be convenient to denote the averaged iterate at the $t$th iteration as:
\beq \textstyle 
\Bprm^t := (1/n) \sum_{i=1}^n \prm_i^t .
\eeq
Observe the following basic convergence results for {\sf DSGD}:
\begin{theorem}\label{thm1}
Under A\ref{ass: graph}--\ref{ass:hete}, suppose that there exists $b \in \RR_+$ such that
\[
\textstyle \sup_{t\geq 1}\gamma_{t} \leq \min\left\{
    \frac{\rho}{4L}, \sqrt{\frac{\rho /(4b)}{1-\rho /2}}  \right\},~\frac{\gamma_{t}}{\gamma_{t+1}} \leq 1+ b\gamma_{t+1}.
\]
Let ${\sf D}\eqdef f(\Bprm^0)- f^\star$. For any $T\geq 1$, it holds 
\begin{align}
& \textstyle \EE\left[\sum_{t=0}^{T-1}\gamma_{t+1}\norm{\grd f(\Bprm^t)}^2 \right] \\
& \leq 4 {\sf D} + \frac{2L\sigma^2}{n} \sum_{t=0}^{T-1}\gamma_{t+1}^2 +  \frac{32L^2(\varsigma^2+\sigma^2)}{\rho^2} \sum_{t=0}^{T-1}\gamma_{t+1}^3. \nonumber
\end{align}
\end{theorem}
\noindent  
Note that the condition on $\{ \gamma_t \}_{t\geq 1}$ can be satisfied by a diminishing step size sequence, or a constant step size. 
The results in the theorem can be simplified as:
\begin{Corollary}\label{cor:1}
For any $T \geq 1$, set $\gamma_{t+1} = 1 / \sqrt{T}$ and let ${\sf T}$ be an r.v.~chosen uniformly from $\{0,\ldots, T-1\}$. Under A\ref{ass: graph}--\ref{ass:hete}, 
\begin{align}
\hspace{-.1cm}\EE\left[ \norm{\grd f(\Bprm^{\sf T})}^2 \right] & = {\cal O}\left( \frac{{\sf D}+L\sigma^2/n}{\sqrt{T}} +  \!\frac{L^2(\varsigma^2+\sigma^2)}{\rho^2 T} \right). \label{eq:corNor}
\end{align}
\end{Corollary}
\noindent The first term $\propto {\sf D}+ n^{-1} L\sigma^2$ is identical to the convergence rate for {\sf CSGD} using a minibatch size of $n$, and it decays with respect to (w.r.t.) the iteration number as ${\cal O}(1/\sqrt{T})$; {see \cite{ghadimi2013stochastic}}. On the other hand, the second term $\propto L^2(\varsigma^2+\sigma^2) / \rho^2$ accounts for the effect of the communication network, and it decays w.r.t.~the iteration number as ${\cal O}(1/{T})$. 

The difference in timescales for the two terms in \eqref{eq:corNor} leads to an intriguing observation. {\sf DSGD} exhibits a behavior that corresponds to a \emph{transient time} characterization in terms of the convergence rate. If the iteration number $T$ satisfies
\beq\label{b1}
T \geq {\cal O} \left(\frac{L^4(\varsigma^4 + \sigma^4)}{\rho^4({\sf D}+L\sigma^2/n)^2}\right),
\eeq 
then $\EE\left[ \| {\grd f(\Bprm^{\sf T})} \|^2 \right] = {\cal O} ( ({\sf D}+ n^{-1} L\sigma^2 ) / {\sqrt{T}} )$, where ${\sf T} \sim {\cal U} \{ 0, \ldots, T-1 \}$. In other words, for large enough iteration number, {\sf DSGD} enjoys a similar convergence rate as its centralized counterpart with a minibatch size of $n$. 

However, a pitfall in the transient time analyzed in Corollary \ref{cor:1} is its poor dependence on the network size. In fact, the transient time in \eqref{b1} can be as large as ${\cal O} ( n^2 / \rho^4 )$ when $n \gg 1$, $\rho \ll 1$. Furthermore, we observe that this growth of the transient time remains unaffected even if the data across agents are completely homogeneous  with $\varsigma = 0$.
The above motivates the current paper to consider a tighter bound for {\sf DSGD}. In particular, with a finer grained analysis, our results will show that data homogeneity plays an important role in reducing the transient time.


\section{Main Results}
This section introduces the main result of this paper on deriving an accelerated convergence rate for {\sf DSGD} which can leverage data homogeneity across agents. 

We preface the main technical results by describing a simple case study to illustrate a key insight.
Consider the following special case of \eqref{eq:opt} with:
\beq \label{eq:quadratic}
f_i( \prm ) = (1/2) \prm^\top {\bm A} \prm + \prm^\top {\bm b},
\eeq 
where ${\bm A} \in \RR^{d \times d}$ is a positive definite matrix and ${\bm b} \in \RR^d$ is a fixed vector. 
Notice that the same ${\bm A}, {\bm b}$ are shared among the agents, indicating that the data held by agents are homogeneous\footnote{It is also possible to consider a relaxed model with heterogeneous ${\bm b}_i$ as our analysis relies on a weaker form of data homogeneity.}. 

Consider the stochastic gradient map where $z_i \equiv \tilde{\bm b}_i \sim {\sf B}_i \equiv {\sf B}$ satisfies
\beq \label{eq:stoc_grad_quad}
\grd \ell ( \prm ; z_i ) = {\bm A} \prm + \widetilde{\bm b}_i,
\eeq 
and $\widetilde{\bm b}_i$ is an independent r.v.~with $\EE[ \widetilde{\bm b}_i ] = {\bm b}$ and bounded variance $\EE[ \| \widetilde{\bm b}_i - {\bm b} \|^2 ] \leq \sigma^2$. Note that this clearly implies $\EE[ \| \grd \ell ( \prm ; z_i ) - \grd f_i (\prm) \|^2 ] \leq \sigma^2$ for any $\prm \in \RR^d$. 

With \eqref{eq:stoc_grad_quad}, the {\sf DSGD} algorithm reads:
\begin{align} 
& \textstyle \prm_i^{t+1} = \sum_{j=1}^n W_{ij} \prm_i^t - \gamma_{t+1} \big( {\bm A} \prm_i^t  + \widetilde{\bm b}_i \big) \label{eq:dsgd_quad} 
\end{align} 
Taking the average over $i=1,\ldots,n$ implies 
\beq \label{eq:linearkey}
\begin{split} 
\Bprm^{t+1} 
& \textstyle  = \Bprm^t - {\gamma_{t+1}} \big( {\bm A} \Bprm^t  + \sum_{i=1}^n \widetilde{\bm b}_i  /n \big)
\end{split}
\eeq 
We observe that the last term is an \emph{unbiased estimate} of the global gradient in \eqref{eq:opt} with $ \grd f( \Bprm^t ) = \EE[ {n}^{-1} \sum_{i=1}^n \big( {\bm A} \prm_i^t  + \widetilde{\bm b}_i \big) ]$ with the variance
\beq 
\textstyle \EE[ \| {\bm A} \Bprm^t + {n}^{-1} \sum_{i=1}^n \widetilde{\bm b}_i -  \grd f( \Bprm^t ) \|^2 ] \leq n^{-1} \sigma^2.
\eeq
Subsequently, the {\sf DSGD} recursion \eqref{eq:dsgd_quad} of the averaged iterate $\Bprm^t$ behaves identically as a centralized SGD algorithm that draws $n$ independent samples of stochastic gradient per iteration using \eqref{eq:stoc_grad_quad}. In other words, for this special case, the transient time of {\sf DSGD} shall be \emph{zero} as the latter matches the performance of {\sf CSGD} \emph{exactly}.

The above case study indicates that {\sf DSGD} may be able to leverage data homogeneity across the agents for accelerating its convergence. In particular, we anticipate the \emph{transient time} of {\sf DSGD} to be much faster if data distributions of agents are close to each other; cf.~$\varsigma \approx 0$.



\subsection{Convergence of {\sf DSGD}: Accelerated Rate}
To derive an improved bound for {\sf DSGD}, we consider the following set of additional assumptions.   

\begin{Assumption} \label{ass:Hsmooth}
For any $i=1,\ldots, n$, there exists $L_H \geq 0$ such that 
\beq
\| \grd^2 f_i( \prm' ) - \grd^2 f_i( \prm ) \| \leq L_{H} \| \prm' - \prm \|,~\forall~\prm, \prm' \in \RR^d.
\eeq
\end{Assumption} 
\noindent Notice that A\ref{ass:Hsmooth} requires  the Hessian of each $f_i$ to be Lipschitz continuous, i.e., $f_i$ is twice continuously differentiable. For quadratic functions, we observe that $L_H = 0$.  

\begin{Assumption}\label{ass:hete2}
There exists $\varsigma_H \geq 0$ such that for any $i=1,\ldots, n$,
\begin{align}
& \| \grd^2 f( \prm ) - \grd^2 f_i( \prm) \| \leq \varsigma_H,~\forall~\prm \in \RR^d. \label{eq:hete_hess}
\end{align}
\end{Assumption}
\noindent The above condition requires the Hessians of the component function $f_i(\prm)$ to be bounded from each other. While both A\ref{ass:hete}, A\ref{ass:hete2} impose conditions on the data homogeneity, we remark that having $\varsigma_H = 0$ in A\ref{ass:hete2} is strictly weaker than having $\varsigma = 0$ in A\ref{ass:hete} as the latter implies the former but not vice versa. Having $\varsigma_H = 0$ \emph{only} requires the quadratic (or higher order) terms of $f_i, f$ to be equal.
We remark that this has been shown to be a critical condition for accelerating distributed optimization \cite{arjevani2015communication, tian2021acceleration}.
Furthermore, under A\ref{ass: smooth}, it is known that $\varsigma_H \leq 2L$. 


Lastly, we strengthen A\ref{ass:SecOrdMom} to a $4$th order moment bound on the oscillation of stochastic gradients. 
\begin{Assumption}\label{ass:SecOrdMom2}
For any $i=1,\ldots, n$ and fixed $\prm \in \RR^d$, it holds $\EE_{z \sim {\sf B}_i} [ \grd \ell ( \prm ; z ) ] = \grd f_i (\prm)$ and there exists $\sigma \geq 0$ with
\begin{align}
\EE_{z \sim {\sf B}_i} [ \| \grd \ell ( \prm ; z ) - \grd f_i( \prm ) \|^4 ] \leq \hsigma^4. \label{A4-2}
\end{align}
\end{Assumption}
\noindent To simplify notations, we have taken the same constant $\sigma$ for the variance bounds in A\ref{ass:SecOrdMom} and A\ref{ass:SecOrdMom2}.


Under the refined conditions, we obtain the following improved convergence rate for the {\sf DSGD} algorithm:
\begin{theorem}\label{thm2}
Under A\ref{ass: graph}--A\ref{ass:SecOrdMom2}, suppose that there exists $b \in \RR_+$ such that
\[
\sup_{t\geq 1}\gamma_{t} \!\leq\!  \min\left\{\frac{\rho}{9L}, 
\sqrt[2]{\frac{\rho /(4b)}{1-\rho /2}} ,  \sqrt[4]{\frac{\rho /(4b)}{1-\rho /2}} \right\},
\]
and ${\gamma_t^4}/\gamma_{t+1}^{4}\! \leq\! 1+b\gamma_{t+1}^4$. Let ${\sf D} = f( \Bprm^0 ) - f^\star$. For any $T \geq 1$, it holds 
\begin{align*}
& \EE \left[ \sum_{t=0}^{T-1}\gamma_{t+1}\norm{\grd f(\Bprm^t)}^2 \right] \leq 4 {\sf D} + \frac{2L\sigma^2}{n} \sum_{t=0}^{T-1} \gamma_{t+1}^2 \\
& + \frac{432L_H^2}{\rho^4}(\sigma^4+4\varsigma^2) \sum_{t=0}^{T-1}\gamma_{t+1}^5 + \frac{32 \varsigma_H^2 (\sigma^2 + \varsigma^2) }{\rho^2} \sum_{t=0}^{T-1} \gamma_{t+1}^3.
 \end{align*}
\end{theorem}
\noindent 
Observe that the step size conditions are similar to \Cref{thm1} and can therefore be satisfied with constant or diminishing step sizes. Furthermore, we have the following corollary:
\begin{Corollary}\label{cor:2}
For any $T \geq 1$, set $\gamma_{t+1} = 1 / \sqrt{T}$ and let ${\sf T}$ be chosen uniformly at random from $\{0,\ldots, T-1\}$. Under A\ref{ass: graph}--\ref{ass:SecOrdMom2} with $\varsigma_H = 0$, it holds
\begin{align}
\hspace{-.2cm}\EE \left[ \norm{\grd f(\Bprm^{\sf T})}^2 \right] \!=\! {\cal O} \left( \frac{{\sf D}+L\sigma^2/n}{\sqrt{T}} + \frac{ L_H^2(\sigma^4 + \varsigma^4)}{\rho^4 T^{2} / n} \right)\!.\!\label{eq:corHOS}
\end{align}
\end{Corollary}
\noindent Notice that we have concentrated on the scenario when $\varsigma_H = 0$ to highlight on the effect when the data across agents are homogeneous in light of A\ref{ass:hete2}. 

The bound in \eqref{eq:corHOS} can be interpreted as follows. The first term $\propto {\sf D}+ n^{-1} L\sigma^2$ is the same term in the convergence of {\sf CSGD}; the second term $\propto n L_H^2(\sigma^4 + \varsigma^4) / \rho^4$ is the communication network-dependent term. Similar to \eqref{eq:corNor}, we observe a difference of the timescales for the two terms w.r.t.~the iteration number $T$. {The first and second term decays at a rate of ${\cal O}(1/\sqrt{T})$, ${\cal O}(1/T^2)$, respectively}. 

Performing a similar calculation to \eqref{b1} gives an \emph{improved transient time} for {\sf DSGD}:
if the iteration number satisfies 
\beq \label{b2}
T \geq {\cal O} \left( \frac{L_H^{\frac{4}{3}} (\sigma^\frac{8}{3} + \varsigma^\frac{8}{3}) n^{\frac{2}{3}} } {\rho^\frac{8}{3} ({\sf D}+ L\sigma^2/n )^\frac{2}{3}} \right), 
\eeq
then {\sf DSGD} enjoys a similar convergence rate as its centralized counterpart with a minibatch size of $n$. Now, if $n \gg 1$, $\rho \ll 1$, the transient time can be simplified to ${\cal O} ( n^{4/3} / \rho^{8/3} )$ which has a better scaling w.r.t.~$n, \rho$ than \eqref{b1} even if the condition $\varsigma = 0$ is enforced in the latter for the data homogeneity assumption A\ref{ass:hete}. 

Our convergence analysis demonstrates that data homogeneity plays an important role in accelerating the transient time of {\sf DSGD}. We remark that if $\varsigma_H$ is far from zero, then the acceleration observed in \eqref{b2} is no longer valid. In our analysis that led to \Cref{thm2}, a key observation is to exploit the high order smoothness property [cf.~A\ref{ass:Hsmooth}] that allowed us to approximate the local gradients via a linear map. Subsequently, we obtain a similar form to \eqref{eq:linearkey} for the {\sf DSGD} recursion and thus an improved convergence rate.

To our best knowledge, this is the first analysis to explicitly account for data homogeneity using the high order smoothness condition. We show that the latter leads to an improved convergence rate for the plain {\sf DSGD} algorithm.


\section{Proof Outline}\label{sec:outline}
This section outlines the proofs of \Cref{thm1}, \ref{thm2}. To simplify notations, we define the $n \times d$ matrices:
\[
    \Prm^{t}\eqdef \left(\prm_1^{t},\prm_2^{t},\cdots, \prm_n^{t} \right)^\top,~
    \BPrm^{t}\eqdef \left(\Bprm^{t},\Bprm^{t},\cdots, \Bprm^{t}\right)^\top 
\] 
and $\CSE{t} \eqdef \Prm^{t}-\BPrm^{t}$ denotes the consensus error, for any $t \geq 0$. Furthermore, we denote $\EE_t[\cdot]$ as the expectation operator conditioned on the random variables up to $t$th iteration. 

Using the above notations, \eqref{algo} implies the following update recursion for the average iterate of {\sf DSGD}:
\beq \textstyle \label{eq:dsgd_avg}
\Bprm^{t+1} = \Bprm^t - \gamma_{t+1} \sum_{i=1}^n \grd \ell( \prm_i^t ; z_i^{t+1} ) / n,
\eeq 
note that $\EE_t[ \sum_{i=1}^n \grd \ell( \prm_i^t ; z_i^{t+1} ) / n ] = \sum_{i=1}^n \grd f_i( \prm_i^t ) / n$.
Observe that the only difference between \eqref{eq:dsgd_avg} and {\sf CSGD} is in the second term. Particularly, {\sf CSGD} would use 
\beq \textstyle 
\Bprm^{t+1} = \Bprm^t - \gamma_{t+1} \sum_{i=1}^n \grd \ell( \Bprm^t ; z_i^{t+1} ) / n,
\eeq 
As we shall demonstrate next, the proofs for \Cref{thm1}, \ref{thm2} differ in how we account for the above deviation between {\sf CSGD} and {\sf DSGD}. 
\vspace{.2cm}

\noindent \emph{Proof of \Cref{thm1}.} Although the analysis for \Cref{thm1} is standard and can be found, e.g., in \cite{lian2017decentralized}, we shall discuss its proof briefly to highlight its difference with \Cref{thm2}. We begin by observing the following descent lemma:
\begin{lemma}[Basic Descent Lemma]\label{lem:aa}
    Under A\ref{ass: graph}-\ref{ass:hete}, if $\sup_{t \geq 1} \gamma_{t} \leq\frac{1}{4L}$, then for any $t \geq 0$, it holds
    \begin{align} \label{eq:aa}
        \EE_t[ f( \Bprm^{t+1} ) ] &\leq  f( \Bprm^t ) - \frac{\gamma_{t+1}}{4} \| \grd f( \Bprm^t ) \|^2 + \frac{ \gamma_{t+1}^2 L \sigma^2}{2n} \\
        &+ \gamma_{t+1} \frac{L^2}{n} \norm{ \CSE{t} }^2_F. \nonumber
    \end{align}
\end{lemma}
\noindent We highlight that \eqref{eq:aa} was derived through using the smoothness property A\ref{ass: smooth} to handle the difference $\sum_{i=1}^n \grd f_i( \prm_i^t) - \grd f_i( \Bprm^t)$, which is proportional to $L^2 \norm{\CSE{t+1}}_F^2$.

The above lemma prompts us to bound the consensus error $\norm{\CSE{t+1}}_F^2$. A key observation is:
\begin{lemma}[Consensus Error Bound]\label{lem:ceb}
Under A\ref{ass: graph}, A\ref{ass:SecOrdMom},  A\ref{ass:hete}, if $\sup_{t \geq 1}\gamma_{t} \leq \frac{\rho}{4L}$,
and there exists a constant $b\in\RR^+$ such that $\gamma_t^2/ \gamma_{t+1}^2\leq 1+b \gamma_{t+1}^2$ for all $t$, and $\prm_i^0=\prm_j^0$ for all $i,j$, then it holds
\begin{align}
    \EE\left[ \norm{\CSE{t+1}}_F^2 \right]  & \leq   \frac{8n}{\rho^2}\gamma_{t+1}^2\left(\varsigma^2 + \sigma^2\right). \label{eq:ceb}
\end{align}
\end{lemma}
\noindent 
Substituting \eqref{eq:ceb} back into \eqref{eq:aa}, 
and summing up from $t=0$ to $T-1$ lead to
the bound for \Cref{thm1}.\hfill~$\square$\vspace{.2cm}

\noindent \emph{Proof of \Cref{thm2}.} The key observation made in our proof is the following property. Define the linear map approximation error for the gradient map $\grd f_i$ as:
\beq 
{\mathscr{E}_i}( \prm'; \prm ) := \grd f_i( \prm' ) - \grd f_i( \prm ) - \grd^2 f_i ( \prm ) ( \prm' - \prm ).
\eeq 
It holds that
\beq \label{eq:nesterov}
\| \mathscr{ E}_i( \prm'; \prm ) \|  \leq \frac{L_{H}}{2} \| \prm' - \prm \|^2,~\forall~\prm', \prm \in \RR^d.
\eeq
The above is due to the high order smoothness condition A\ref{ass:Hsmooth}; see \cite[Lemma 1.2.5]{nesterov2003introductory}. 
It inspires us to consider the following relation:
\begin{align}
& \textstyle \frac{1}{n} \sum_{i=1}^n \{ \grd f_i( \prm_i^t) - \grd f_i( \Bprm^t ) \} \label{eq:linear_approx} \\
& \textstyle = \frac{1}{n} \sum_{i=1}^n \{ \grd^2 f_i( \Bprm^t ) ( \prm_i^t - \Bprm^t ) + \mathscr{E}_i( \prm_i^t ; \Bprm^t ) \}. \nonumber \\
& \textstyle = \frac{1}{n} \sum_{i=1}^n \{ ( \grd^2 f_i( \Bprm^t ) - \grd^2 f( \Bprm^t ) ) ( \prm_i^t - \Bprm^t ) + \mathscr{E}_i( \prm_i^t ; \Bprm^t ) \}, \nonumber 
\end{align}
where the last equality is due to $(1/n) \sum_{i=1}^n \grd^2 f( \Bprm^t ) \prm_i^t = \grd^2 f( \Bprm^t ) \Bprm^t$ since the map is linear. 

The last equality in \eqref{eq:linear_approx} enables a fine grained analysis on the difference in mean fields between the updates used in {\sf DSGD} and {\sf CSGD}. In fact, we obtain the following bound:
\beq \nonumber 
\norm{ \sum_{i=1}^n \frac{ \grd f_i( \prm_i^t) - \grd f_i( \Bprm^t ) }{n} }^2 \leq \left( \frac{ L_{H}^2}{2n} \| \CSE{t} \|_F^2 
+ \frac{\varsigma_H^2}{n}  \right) \| \CSE{t} \|_F^2.
\eeq 
Note that with A\ref{ass: smooth} alone, the bound would be just $L^2 \norm{\CSE{t}}_F^2 / n$. In contrast, we obtained a finer bound with $4$th order consensus error through the high order smoothness condition.
To conclude, the above insights gives an improved descent lemma upon Lemma~\ref{lem:aa}: 
\begin{lemma}[Improved Descent Lemma]\label{lem:des}
    Under A\ref{ass: graph}, A\ref{ass: smooth}, A\ref{ass:SecOrdMom}, A\ref{ass:Hsmooth}, A\ref{ass:hete2}, if $ \sup_{t \geq 1} \gamma_{t} \leq \frac{1}{4L}$, then for any $t \geq 0$, it holds
    \beq \label{eq:des}
    \begin{aligned}
        \EE_t[ f( \Bprm^{t+1} ) ] & \leq f( \Bprm^t ) - \frac{\gamma_{t+1}}{4} \| \grd f( \Bprm^t ) \|^2 + \frac{ \gamma_{t+1}^2 L \sigma^2}{2n} \\
        & \quad + \frac{\gamma_{t+1}}{2n} \Big(  L_{H}^2 \| \CSE{t} \|_F^{2} + 2 \varsigma_H^2  \Big) \| \CSE{t} \|_F^{2}.
    \end{aligned}
    \eeq
\end{lemma}
\noindent Observe \eqref{eq:des} differs from \eqref{eq:aa} only by the last term proportional to the consensus error $\norm{\CSE{t}}_F^2$. 
Furthermore, the high order consensus error admits the following bound:
\begin{lemma}[High Order Consensus Error Bound]\label{lem:ceb-h} 
Under A\ref{ass: graph}, A\ref{ass: smooth}, A\ref{ass:hete}, A\ref{ass:SecOrdMom2}, if $ \sup_{t \geq 1}\gamma_{t}\leq \frac{\rho}{9 L} $, and there exist a constant $b\in\RR^+$ such that $\gamma_t^4/\gamma_{t+1}^{4}\leq 1+b\gamma_t^4$, and $\prm_i^0=\prm_j^0$ for all $i,j$, then for any $t \geq 0$, it holds
    \begin{align}
        \EE \left[ \norm{\CSE{t+1}}_F^4 \right] & \leq \frac{216(\sigma^4  + 4  \varsigma^4) n^2 }{\rho^4}\gamma_{t+1}^4. \label{eq:ceb-h}
    \end{align}
\end{lemma}
\noindent Substituting Lemmas \ref{lem:ceb}, \ref{lem:ceb-h} into Lemma \ref{lem:des}, 
and  taking summation from $t=0$ to $T-1$ for the both sides, then 
\beqq
\begin{aligned}
    &\sum_{t=0}^{T-1}\gamma_{t+1}\norm{\grd f(\Bprm^t)}^2 \leq 4\left[f(\Bprm^0)- f^\star \right] + \frac{2L\sigma^2}{n}\sum_{t=0}^{T-1}\gamma_{t+1}^2 
    \\
    &+ \frac{432 L_H^2 (\sigma^4+4\varsigma^2)}{\rho^4 / n} \sum_{t=0}^{T-1}\gamma_{t+1}^{5} + \frac{32 \varsigma_H^2 (\sigma^2 + \varsigma^2) }{\rho^2} \sum_{t=0}^{T-1} \gamma_{t+1}^3,
\end{aligned}
\eeqq
which completes the proof of Theorem \ref{thm2}. \hfill $\square$

\section{Numerical Experiments}
This section presents preliminary experiment to verify accelerated convergence with homogeneous data using {\sf DSGD}.

\begin{figure}[t]
    \centering
    \includegraphics[scale=0.45]{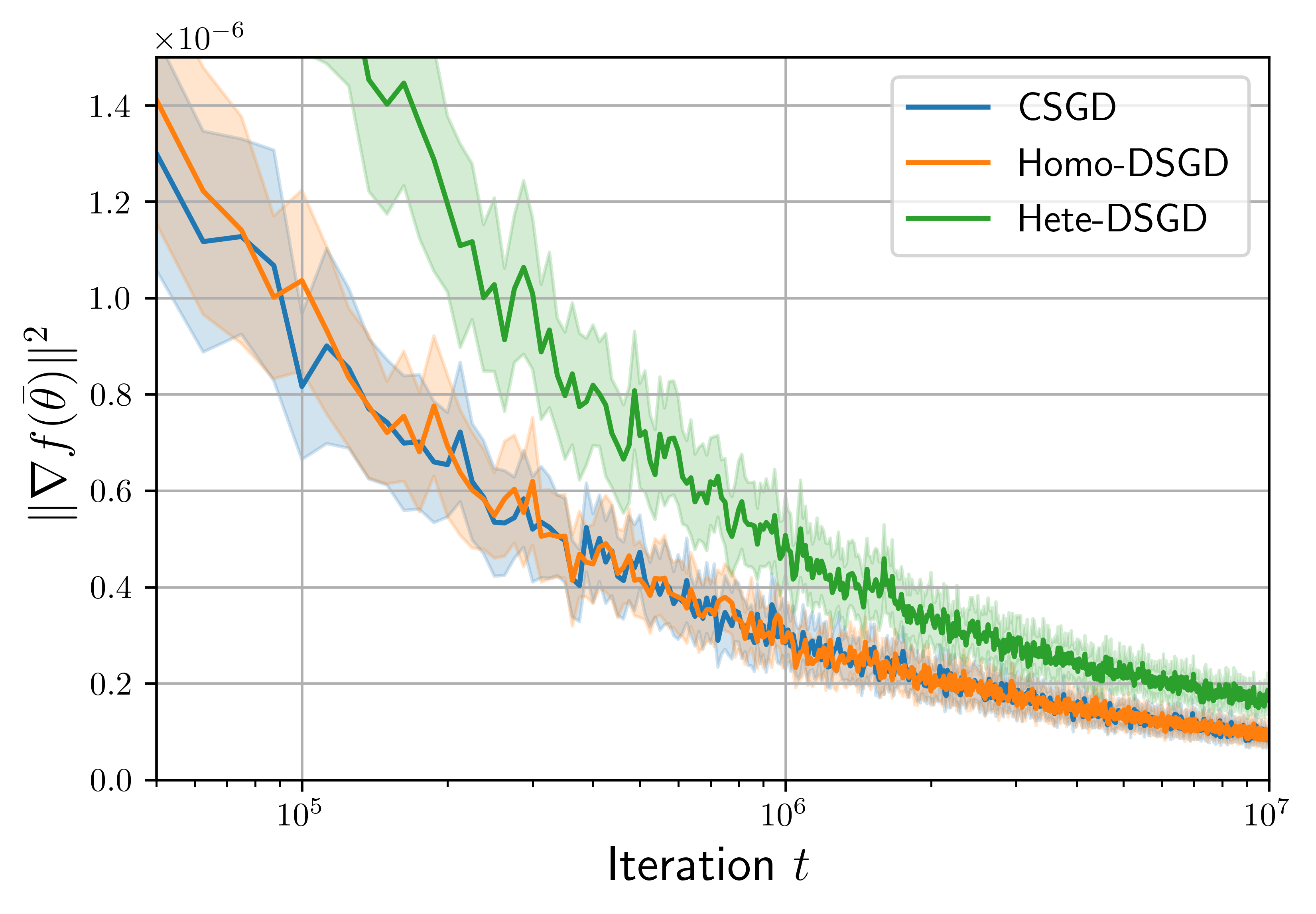}\vspace{-.2cm}
    \caption{Comparing the convergence behavior of {\sf DSGD} algorithm under homogeneous / heterogeneous data alongside with {\sf CSGD}. At each iteration, each agent simulated in the {\sf Homo-DSGD}/{\sf Hete-DSGD} algorithms take $1$ sample, while the {\sf CSGD} algorithm takes a batch of $n$ samples.}\vspace{-.3cm}
    \label{fig1}
\end{figure}

Our aim is to verify \Cref{thm1}, \ref{thm2} via comparing the performance of {\sf DSGD} under homogeneous and heterogeneous data. Consider a binary classification problem using a linear model for \eqref{eq:opt}, \eqref{eq:stocfct}, where the loss function takes the form of a non-convex sigmoid function:
\beq \label{eq:sigmoid} \textstyle 
\ell( \prm; z ) = ({1 + \exp( y \pscal{x}{\prm}) })^{-1} + \frac{\beta}{2} \norm{\prm}^2 ,
\eeq 
where $z \equiv (x, y)$ with the feature $x \in \RR^d$, the label $y \in \{ \pm 1 \}$ represents a (training) data sample and $\beta > 0$ is a regularization parameter. Note that \eqref{eq:sigmoid} satisfies A\ref{ass: smooth}--A\ref{ass:SecOrdMom2}.
We consider $n=12$ agents connected on a ring graph, with a self weight of $W_{ii} = 0.9$. 

We set $d=5$ and $\beta=10/m$, where $m$ is the total number of data samples at the agents. To simulate the \emph{heterogeneous data} setting ($\varsigma, \varsigma_H \neq 0$), we first generate the parameter vectors as $\prm_{\text{o},i}\in {\cal U}[-1 + \frac{i-1}{6}, -1+\frac{i}{6}]^{5} $ , $i=1,\ldots, 12$. Then, for each $i=1,\ldots,12$, the data distribution ${\sf B}_i$ is taken to be the empirical distribution of $n_i = 200$ samples $\{ x_j^i , y_j^i \}_{j=1}^{200}$, which are generated as 
\[ 
x_j^i \sim {\cal U}[-1,1]^{5},~~ y_j^i = \sign(\Pscal{x_j^i}{\prm_{\text{o}, i}}).
\]
Subsequently, we denote the algorithm where agent $i$ draws samples from the above ${\sf B}_i$ as the {\sf Hete-DSGD} algorithm. 
On the other hand, to simulate the \emph{homogeneous data} setting ($\varsigma, \varsigma_H = 0$), we consider a combined dataset by taking $\tilde{\sf B}_i$ to be the empirical distribution of $\{ \{ x_j^i , y_j^i \}_{j=1}^{200} \}_{i=1}^{12}$ such that ${\sf B} = \tilde{\sf B}_i$ for all $i$. The corresponding algorithm is denoted as the {\sf Homo-DSGD} algorithm. 
As benchmark, we also consider the {\sf CSGD} algorithm which draws a minibatch of $n=12$ samples from ${\sf B}$ at each iteration. 
Furthermore, we set the stepsize in the algorithms as $\gamma_{t+1}= \sqrt{a_0/(a_1 + t)}$ with 
$a_0=1/\beta$, $a_1=8 L^2 / \beta^2$. 
For the {\sf Homo-DSGD}/{\sf Hete-DSGD} algorithms, the initial solution for each agent is randomly drawn as $\prm_i^{(0)} \sim {\cal N}(1, 0.8)^{5}$, and we use the average $\Bprm^{(0)}$ as the initial solution for {\sf CSGD}.



In Fig.~\ref{fig1}, we compare the norm of gradient $\|\grd f(\Bprm^t)\|^2$ against the number of iteration $t$ for the tested algorithms, over $50$ repeated runs of the stochastic algorithms. The shaded region indicate the 95\% confidence interval. Observe that the two {\sf DSGD} algorithms approach the same steady state convergence behavior as the centralized algorithm {\sf CSGD} as $t \to \infty$, validating our basic result in \Cref{thm1}. 
Moreover, we observe that the {\sf Homo-DSGD} algorithm matches the performance of {\sf CSGD} with a much smaller \emph{transient time} than the {\sf Hete-DSGD} algorithm. The observation corroborates with \Cref{thm2}.


\section{Conclusions}
In this work, we provided a fine grained analysis for the convergence rate of {\sf DSGD} while focusing on the role of data homogeneity. Particularly, we show that the plain {\sf DSGD} algorithm may achieve fast convergence when the data distribution across agents are similar to each other. Our theoretical results are supported by numerical experiment.

\bibliographystyle{IEEEtran}
\bibliography{ecl.bib}






\appendices
\section{Proof of Lemma \ref{lem:aa}}
Throughout the appendix, we shall use the notations $\grd f_i^t := \grd f_i(\prm_i^t)$, $\Tgrd f_i^t := \grd \ell( \prm_i^t ; z_i^{t+1} )$, and
\[
\Tgrd F^t \eqdef \left(\Tgrd f_1^t, \Tgrd f_2^t, \cdots, \Tgrd f_n^t\right)^\top \in \RR^{n \times d}.
\]

Turning our attention back to the proof of Lemma~\ref{lem:aa}. Using A\ref{ass: smooth} and \eqref{eq:dsgd_avg}, we have{\small 
\begin{align*}
f( \Bprm^{t+1} ) \leq  f( \Bprm^t )  - \langle \grd f( \Bprm^t ) | {\textstyle \frac{ \gamma_{t+1}}{n}} \! \sum_{i=1}^n \! \Tgrd f_i^t\rangle  +   {\textstyle \frac{\gamma_{t+1}^2 L}{2 n^2} }\left\| \sum_{i=1}^n \Tgrd f_i^t \right\|^2.
\end{align*}}Taking the conditional expectation $\EE_t[\cdot]$ on both sides yields
\begin{align} \label{eq:expF}
 \!\EE_t[ f( \Bprm^{t+1} ) ] &\leq f( \Bprm^t ) \!-\! \gamma_{t+1} \textstyle \Pscal{ \grd f( \Bprm^t ) }{ \frac{1}{n} \sum_{i=1}^n \grd f_i^t }\\
&\textstyle \quad + \gamma_{t+1}^2 \frac{L}{2} \EE_t \left[ \left\| \frac{1}{n} \sum_{i=1}^n \Tgrd f_i^t \right\|^2 \right]. \nonumber
\end{align}
We note that the second term is lower bounded by
\begin{align*}
&\textstyle \pscal{ \grd f( \Bprm^t ) }{ \frac{1}{n} \sum_{i=1}^n \grd f_i^t } \\
&\geq \textstyle \frac{1}{2} \| \grd f( \Bprm^t ) \|^2- \frac{1}{2} \left\| \frac{1}{n} \sum_{i=1}^n \big( \grd f_i( \Bprm^t) - \grd f_i^t \big) \right\|^2 ,
\end{align*}
and the last term can be upper bounded by
\begin{align*}
& \textstyle \EE_t \left[ \left\| \frac{1}{n} \sum_{i=1}^n \Tgrd f_i^t \right\|^2 \right] 
\leq \frac{\sigma^2}{n} + \left\| \frac{1}{n} \sum_{i=1}^n \grd f_i^t \right\|^2 \\
& \textstyle \leq \frac{\sigma^2}{n} + 2 \| \grd f( \Bprm^t) \|^2 + 2 \left\| \frac{1}{n} \sum_{i=1}^n \big( \grd f_i( \Bprm^t) - \grd f_i^t \big) \right\|^2,
\end{align*}
where we have used A\ref{ass:SecOrdMom} in the first inequality. Substituting into \eqref{eq:expF} 
and using the step size condition $\gamma_{t+1} \leq \frac{1}{4L}$ gives
\beq \label{eq:expF1}
\EE_t[ f( \Bprm^{t+1} ) ] \leq f( \Bprm^t ) - \frac{\gamma_{t+1}}{4} \| \grd f( \Bprm^t ) \|^2 + \frac{ \gamma_{t+1}^2 L \sigma^2}{2n}\\
\textstyle + \gamma_{t+1} \left\| \frac{1}{n} \sum_{i=1}^n \big( \grd f_i( \Bprm^t) - \grd f_i( \prm_i^t) \big) \right\|^2 \nonumber
\eeq
Observe that by A\ref{ass: smooth}, the last term is bounded by:
\begin{align*}
& \textstyle \left\| \frac{1}{n} \sum_{i=1}^n \big( \grd f_i( \Bprm^t) - \grd f_i^t \big) \right\|^2 
\leq \frac{L^2}{n} \norm{\BPrm^t-\Prm^t}^2_F.
\end{align*}
Substituting back into \eqref{eq:expF1} leads to
\begin{align}
\EE_t[ f( \Bprm^{t+1} ) ] &\leq f( \Bprm^t ) - \frac{\gamma_{t+1}}{4} \| \grd f( \Bprm^t ) \|^2 + \frac{ \gamma_{t+1}^2 L \sigma^2}{2n} \\
&+ \gamma_{t+1} \frac{L^2}{n} \norm{\BPrm^t-\Prm^t}^2_F. \nonumber
\end{align}
This concludes the proof.
\ifonlineapp
\else
Proofs of the other lemmas in Sec.~\ref{sec:outline} can be found in the online appendix: \url{http://www.se.cuhk.edu.hk/~htwai/pdf/cdc22_homo.pdf}.
\fi 
\ifonlineapp
\newpage

\section{Proof of Lemma~\ref{lem:ceb}}
We observe the following chain:
\begin{align}
    & \Prm^{t+1}-\BPrm^{t+1} = \left({\bm I}-\frac{1}{n}{\bm 1}{\bm 1}^\top \right)\Prm^{t+1} = {\bm U}{\bm U}^\top \Prm^{t+1} \nonumber \\   
    &={\bm U}{\bm U}^\top \left({\bm W}\Prm^t -\gamma_{t+1}\Tgrd F^t\right) \nonumber \\
    &= {\bm U}{\bm U}^\top \left[{\bm W} \left({\bm U \bm U}^\top + \frac{1}{n}{\bm 1 \bm 1}^\top\right)\Prm^t -\gamma_{t+1}\Tgrd F^t\right] \nonumber \\
    &= {\bm U}\left({\bm U}^\top {\bm W}{\bm U} \right) {\bm U}^\top \Prm^t - \gamma_{t+1}{\bm U}{\bm U}^\top \Tgrd F^t \label{eq:consensuserr}
\end{align}
Using A\ref{ass: graph} shows that $\norm{\CSE{t+1} }_F^2$ is upper bounded by
\begin{align*}
    & (1+\alpha) (1-\rho)^2 \norm{ \CSE{t} }_F^2 \!+\! \left(1+\frac{1}{\alpha}\right) \gamma_{t+1}^2\norm{{\bm U \bm U}^\top \Tgrd F^t}_F^2,
\end{align*}
for any $\alpha > 0$. Setting $\alpha=\frac{\rho}{1-\rho}$ gives
\begin{align}\label{eq:cseb}
    \norm{\CSE{t+1}}_F^2 
    & \leq (1-\rho)\norm{\CSE{t}}_F^2+\frac{\gamma_{t+1}^2}{\rho}  \norm{{\bm U \bm U}^\top \Tgrd F^t}_F^2.
\end{align}
We observe that
\begin{align*}
    & \norm{{\bm U \bm U}^\top \Tgrd F^t}_F^2 = \sum_{i=1}^{n} \norm{\Tgrd f_i^t-\frac{1}{n}\sum_{j=1}^{n}\Tgrd f_j^t}^2
    \\
    &\!=\! \sum_{i=1}^{n} \norm{\Tgrd f_i^t\!-\!\grd f_i^t \!+\! \grd f_i^t \!-\!\frac{1}{n}\sum_{j=1}^{n}\grd f_j^t\!-\!\frac{1}{n}\sum_{j=1}^{n}\left(\Tgrd f_j^t\!-\!\grd f_j^t\right)}^2
\end{align*}
Taking the conditional expectation $\EE_t[\cdot]$ yields 
\begin{align*}
    & \EE_t \left[ \norm{{\bm U \bm U}^\top \Tgrd F^t}_F^2 \right] 
    \leq 2n\sigma^2 + \sum_{i=1}^{n}\norm{\grd f_i^t -\frac{1}{n}\sum_{j=1}^{n}\grd f_j^t}^2
    \\
    &\overset{(a)}{\leq} 2n ( \varsigma^2 + \sigma^2) + \frac{2L^2}{n}\sum_{i=1}^{n} \sum_{j=1}^n \norm{\prm_i^t-\prm_j^t}^2
    \\
    &\leq 2n ( \varsigma^2 + \sigma^2) + 8L^2 \norm{\CSE{t}}_F^2 
\end{align*}
where (a) is obtained by adding and subtracting $\grd f(\prm_i^t)$ and using A\ref{ass: smooth}, A\ref{ass:hete}. Substituting back into \eqref{eq:cseb} yields
\begin{align*}
    & \EE_t [\norm{\CSE{t+1}}_F^2] 
    \! \leq \! (1\!-\!\rho + \frac{8L^2\gamma_{t+1}^2}{\rho}) \norm{\CSE{t}}_F^2 \!+\! \frac{2n \gamma_{t+1}^2}{\rho}\left(\varsigma^2 \!+\! \sigma^2\right).
\end{align*}
Setting $\gamma_{t+1}\leq \frac{\rho}{4L} $ yields that 
\begin{align*}
    & \EE [\norm{\CSE{t+1}}_F^2] \leq \left(1- \frac{\rho}{2} \right) \EE[ \norm{\CSE{t}}_F^2] + \frac{\gamma_{t+1}^2}{\rho}2n\left(\varsigma^2 + \sigma^2\right).
\end{align*}
Solving the recursion and noticing that $\CSE{0}={\bm 0}$ yields
\begin{align*}
    \EE [\norm{\CSE{t+1}}_F^2]  
    &\leq \frac{2n(\varsigma^2 \!+\! \sigma^2)}{\rho}\sum_{i=1}^{t+1}\gamma_i^2 (1-\frac{\rho}{2})^{t+1-i} \\
    & \leq \frac{8n}{\rho^2}\gamma_{t+1}^2\left(\varsigma^2 \!+\! \sigma^2\right).
\end{align*}
where we have applied Lemma \ref{lem:auxil} in the last inequality.

\section{Proof of Lemma~\ref{lem:des}}

We develop the lemma starting from \eqref{eq:expF1} and focus on analyzing the last term therein. Using A\ref{ass:Hsmooth}, we observe that:
\begin{align*}
& \left\| \frac{1}{n} \sum_{i=1}^n \big( \grd f_i( \Bprm^t) - \grd f_i^t \big) \right\|^2  = \bigg\| \frac{1}{n} \sum_{i=1}^n \big( \grd f_i^t - \grd f_i( \Bprm^t)  
\\
&\qquad - \grd^2 f_i( \Bprm^t ) ( \prm_i^t - \Bprm^t)+ \grd^2 f_i( \Bprm^t ) ( \prm_i^t - \Bprm^t) \big) \bigg\|^2 \\
& \leq 2 \left\| \frac{1}{n} \sum_{i=1}^n \big( \grd f_i^t - \grd f_i( \Bprm^t)  - \grd^2 f_i( \Bprm^t ) ( \prm_i^t - \Bprm^t) \big) \right\|^2 \\
&\qquad + 2 \left\| \frac{1}{n} \sum_{i=1}^n \grd^2 f_i( \Bprm^t ) ( \prm_i^t - \Bprm^t) \right\|^2 \\
& \leq \frac{ L_{H}^2}{2n} \, \| \Prm^t - \BPrm^t \|_F^{4} + \frac{2}{n^2} \left\| \sum_{i=1}^n \grd^2 f_i( \Bprm^t ) ( \prm_i^t - \Bprm^t) \right\|^2,
\end{align*}
where the last inequality is due to A\ref{ass:Hsmooth} and \eqref{eq:nesterov}. 
Furthermore, by A\ref{ass:hete2}, we obtain
\begin{align*}
& \quad \!\bigg\|\! \sum_{i=1}^n \grd^2 f_i( \Bprm^t ) ( \prm_i^t - \Bprm^t) \bigg\|^2 \\
&\!\overset{(a)}{=}\! \left\| \sum_{i=1}^n ( \grd^2 f_i( \Bprm^t ) \!-\! \grd^2 f( \Bprm^t ) ) ( \prm_i^t\! -\! \Bprm^t) \right\|^2 \! \leq\! n \varsigma_H^2 \|\Prm^t \!-\! \BPrm^t \|_F^2,
\end{align*}
where (a) is due to $\sum_{i=1}^n \grd^2 f( \Bprm^t ) ( \prm_i^t - \Bprm^t ) = {\bm 0}$. 
This yields
\begin{align*}
& \left\| \frac{1}{n} \sum_{i=1}^n \big( \grd f_i( \Bprm^t) - \grd f_i( \prm_i^t) \big) \right\|^2 \\
& \leq \frac{1}{n} \left( \frac{ L_{H}^2}{2} \|\Prm^t - \BPrm^t \|_F^2 
+ \varsigma_H^2  \right) \|\Prm^t \!-\! \BPrm^t \|_F^2.
\end{align*}
Substituting back into \eqref{eq:expF} leads to
\beqq
\begin{aligned}
& \EE_t[ f( \Bprm^{t+1} ) ] \leq f( \Bprm^t ) - \frac{\gamma_{t+1}}{4} \| \grd f( \Bprm^t ) \|^2 + \frac{ \gamma_{t+1}^2 L \sigma^2}{2n}
\\
& \qquad + \frac{\gamma_{t+1}}{2n} \Big(  L_{H}^2 \|\Prm^t - \BPrm^t \|_F^{2} + 2 \varsigma_H^2  \Big) \|\Prm^t - \BPrm^t \|_F^{2}.
\end{aligned}
\eeqq
This concludes the proof.

\section{Proof of Lemma~\ref{lem:ceb-h}}
Similar to the proof of Lemma \ref{lem:ceb}, using A\ref{ass: graph} and \eqref{eq:consensuserr}, the following holds for any $\alpha, \beta> 0$,
\begin{align*}
\norm{\CSE{t+1}}^4_F & \leq  \bigg[ (1+\alpha) (1-\rho)^2\norm{{\bm U\bm U}^\top \Prm^t}_F^2 \\
&\quad \textstyle + \left(1+\frac{1}{\alpha}\right) \gamma_{t+1}^2\norm{{\bm U \bm U}^\top \Tgrd F^t}_F^2 \bigg]^2
\\
&\leq  (1+\beta)(1+\alpha)^2  (1-\rho)^4\cdot \norm{\CSE{t}}_F^4 \\
&\quad \textstyle + (1+\frac{1}{\beta})(1+\frac{1}{\alpha})^2  \gamma_{t+1}^{4} \norm{{\bm U}{\bm U}^\top\Tgrd F^t}^4_F .
\end{align*}
Setting $\alpha =\beta = \frac{\rho}{1-\rho}$, then $1+1/\alpha = 1+1/\beta = 1/\rho$, which leads to the following
\begin{align}\label{eq:bb}
\norm{\CSE{t+1}}^4_F \leq (1-\rho)\norm{\CSE{t}}_F^4 + \frac{\gamma_{t+1}^4}{\rho^3}\norm{{\bm U}{\bm U}^\top\Tgrd F^t}_F^4.
\end{align}
The second term can be expanded as:
\begin{align*}
&\norm{{\bm U}{\bm U}^\top\Tgrd F^t}_F^4 = \left( \sum_{i=1}^{n}\norm{\Tgrd f_i^t - \frac{1}{n}\sum_{j=1}^{n}\Tgrd f_j^t}^2 \right)^2.
\end{align*}
Adding and subtracting $\grd f_i^t, \frac{1}{n}\sum_{j=1}^n\grd f_j^t$ leads to the upper bound:
\begin{align*}
 \norm{{\bm U}{\bm U}^\top\Tgrd F^t}_F^4 & \leq \textstyle 27 n \sum_{i=1}^{n}\Bigg[ \norm{\frac{1}{n}\sum_{j=1}^{n}(\Tgrd f_j^t -\grd f_j^t)}^4 \\
&\textstyle + \norm{\Tgrd f_i^t - \grd_i^t}^4  + \norm{\grd f_i^t-\frac{1}{n}\sum_{j=1}^{n}\grd f_j^t}^4 \Bigg],
\end{align*}
where we have used $(a+b+c)^4\leq 27(a^4+b^4+c^4)$. Using the Cauchy-Schwarz inequality, we bound the last term in the above inequality as:
    \begin{align*}
            \norm{\frac{1}{n}\sum_{j=1}^{n}(\Tgrd f_j^t\!-\!\grd f_j^t)}^4 
            &\leq \frac{1}{n^4} \left(n\sum_{j=1}^{n} \norm{\Tgrd f_j^t-\grd f_j^t }^2\right)^2
            \\
            &\leq \frac{1}{n} \sum_{j=1}^{n} \norm{\Tgrd f_j^t-\grd f_j^t}^4
    \end{align*}
    By A\ref{ass:SecOrdMom2}, we observe that
    \begin{align*}
        &\EE_{t}\norm{{\bm U}{\bm U}^\top\Tgrd F^t}_F^4 \leq  27n \sum_{i=1}^{n}\bigg[ 2 \sigma^4 + 
        \norm{\grd f_i^t-\frac{1}{n}\sum_{j=1}^{n}\grd f_j^t}^4\bigg]
    \end{align*}
    Adding and subtracting $\grd f(\prm_i^t) = \frac{1}{n}\sum_{j=1}^{n}\grd f_j(\prm_i^t)$ in the last term leads to
    \begin{align*}
        \EE_t\norm{{\bm U}{\bm U}^\top\Tgrd F^t}_F^4 & \leq 27 n \bigg\{2 n \sigma^2 + 8\sum_{i=1}^{n} \norm{\grd f_i^t \!-\! \grd f(\prm_i^t)}^4 \!
        \\
        &\quad +\textstyle  8\sum_{i=1}^{n} \norm{\frac{1}{n}\sum_{j=1}^{n}\grd f_j(\prm_i^t)\!-\!\grd f_j^t)}^4 \bigg\}
    \end{align*}
    Applying A\ref{ass: smooth}, A\ref{ass:hete} yields
    \begin{align*}
    \!\EE_t\norm{{\bm U}{\bm U}^\top\Tgrd F^t}_F^4 \!\leq\! 27 n \!\left\{\! 2 n \sigma^4 \!+\! 8n\varsigma^4 + \frac{8 L^4}{n} \!{\sum_{i, j=1}^{n}\!} \norm{\prm_i^t - \prm_j^t}^4 \!\right\}\!.
    \end{align*}
    Observe that the last terms are upper bounded by the consensus error $2\norm{\BPrm^t-\Prm^t}^4$. We obtain:
    \begin{align*}
        \EE_t\norm{{\bm U}{\bm U}^\top\Tgrd F^t}_F^4 & \leq 54 n^2 \Bigg\{  \sigma^4\! +\! 4 \varsigma^4 \!+\! \frac{64 L^4}{n^2} \norm{\BPrm^t-\Prm^t}_F^4\!\Bigg\}.
    \end{align*}
    Substituting the upper bound of $\norm{{\bm U}{\bm U}^\top\Tgrd F^t}_F^4$  to inequality (\ref{eq:bb}), we have that
    \begin{align*}
        \EE \norm{\CSE{t+1}}^4_F & \leq \left(1 -\rho +\frac{\gamma_{t+1}^4}{\rho^3} 3456 L^4\right)\EE\norm{\CSE{t}}_F^4 
        \\
        &\quad + 54 n^2 (\sigma^4 + 4 \varsigma^4)\frac{\gamma_{t+1}^4}{\rho^3}
    \end{align*}
    With the stepsize condition $\gamma_{t+1}\leq \frac{\rho}{9 L}$, we obtain 
    \begin{align*}
        \EE\norm{\CSE{t+1}}^4_F & \leq \left( 1 -\frac{\rho}{2} \right)\EE\norm{\CSE{t}}_F^4  + 54n^2 (\sigma^4 + 4 \varsigma^4)\frac{\gamma_{t+1}^4}{\rho^3}.
    \end{align*}
    Solving the recursion and noticing that $\prm_i^0=\prm_j^0$ give
    \begin{align*}
        \EE\norm{\CSE{t+1}}^4_F \leq \frac{54n^2 (\sigma^4 \!+\! 4 \varsigma^4)}{\rho^3} \sum_{j=1}^{t+1}\gamma_{j}^4\left(1-\frac{\rho}{2}\right)^{t+1-j}
    \end{align*}
    Applying Lemma \ref{lem:auxil}, we have
    \begin{align*}
        \EE\norm{\CSE{t+1}}^4_F \leq \frac{216(\sigma^4 + 4 \varsigma^4)n^2 }{\rho^4}\gamma_{t+1}^4.
    \end{align*}
    This concludes the proof.

\section{An Auxiliary Lemma}
\begin{lemma}\label{lem:auxil}
Let $\{\gamma_{t}\}_{t\geq 1}$ be an arbitrary sequence of non-negative number and let $b \in \RR^+$ be given, 
\begin{itemize}[leftmargin=*]
    \item If $\frac{\gamma_t^2}{\gamma_{t+1}^2}\leq 1+b \gamma_{t+1}^2$, $\forall t$, and $\sup_{t \geq 1} \gamma_{t} \leq \sqrt{\frac{\rho /4}{b(1-\rho /2)}}$. The following holds:
    \beq\label{lemeq:auxil_od2} 
        \sum_{i=1}^{t+1}\gamma_i^2 (1-\frac{\rho}{2})^{t+1-i} \leq \frac{4}{\rho} \gamma_{t+1}^2
    \eeq
    \item If $\frac{\gamma_t^4}{\gamma_{t+1}^4}\leq 1+b \gamma_{t+1}^4$, $\forall t$, and $\sup_{t \geq 1} \gamma_{t} \leq \sqrt[4]{\frac{\rho /4}{b(1-\rho /2)}}$. The following holds:
    \beq\label{lemeq:auxil_od4}
     \sum_{i=1}^{t+1}\gamma_i^4 (1-\frac{\rho}{2})^{t+1-i} \leq \frac{4}{\rho} \gamma_{t+1}^4
    \eeq
\end{itemize}
\end{lemma}
\begin{proof}
Observe the following chain:
\beqq
\begin{aligned}
      & \sum_{i=1}^{t+1}\gamma_i^2 (1-\frac{\rho}{2})^{t+1-i} \\ &\leq  \sum_{i=1}^{n}\frac{\gamma_{i}^2}{\gamma_{i+1}^2}\cdot \frac{\gamma_{i+1}^2}{\gamma_{i+2}^2}\cdot \frac{\gamma_{t}^2}{\gamma_{t+1}^2}\cdot \gamma_{t+1}^2(1-\frac{\rho}{2})^{t+1-i} 
      \\
      &\leq \gamma_{t+1}^2 \sum_{i=1}^{t+1} \left[ (1+b \gamma_{i+1}^2)(1-\frac{\rho}{2})\right]^{t+1-i} 
      \\
      & \leq \gamma_{t+1}^2\sum_{i=1}^{t+1} (1-\frac{\rho}{4})^{t+1-i} \leq \frac{4\gamma_{t+1}^2}{\rho}
\end{aligned}
\eeqq
where the last inequality is due to the condition $\gamma_{i+1} \leq \sqrt{\frac{\rho /4}{b(1-\rho /2)}}$. This concludes the proof for \eqref{lemeq:auxil_od2}. 

For the fourth order case in \eqref{lemeq:auxil_od4}, we omit the proof due to space limitation. Notice that it follows directly from the above derivation on the second order case.
\end{proof}

\fi

\end{document}